\documentclass[final]{siamltex}

\usepackage{epsfig}
\usepackage{amsmath,amssymb}
\usepackage {graphicx}
\usepackage{amsmath,amssymb}

\newtheorem{example}{Example}
\newtheorem{remark}{Remark}[section]

\title{Refined Bounds on the Number of Distinct Eigenvalues of a Matrix After Perturbation}

\author{Yunjie Wang\thanks{School of Mathematics, China University of Mining and Technology \&
                           Kewen Institute, Jiangsu Normal University, Xuzhou, 221116,  Jiangsu, P.R. China. E-mail: {\tt wangyunjiemath@163.com}.}
       \and Gang Wu\thanks{Corresponding author. School of Mathematics,
China University of Mining and Technology, Xuzhou, 221116, Jiangsu, P.R. China.
E-mail: {\tt gangwu@cumt.edu.cn} and {\tt gangwu76@126.com}. This author is
supported by the National Science Foundation of China under grant 11371176, the Natural Science Foundation of Jiangsu
Province under grant BK20171185, and the Talent Introduction Program of China
University of Mining and Technology.}
}

\begin{document}
\maketitle

\begin{abstract}
The eigenproblem of low-rank updated matrices are of crucial importance in many applications. Recently, an upper bound on the number of distinct eigenvalues of a perturbed matrix was established. The result can be applied to estimate the number of Krylov iterations required for solving a perturbed
linear system.
In this paper, we revisit this problem and establish some refined bounds. Some {\it a prior} upper bounds that only rely on the information of the matrix in question and the low-rank update are provided. Examples show the superiority of our theoretical results over the existing ones.
The number of distinct singular values of a matrix after perturbation is also investigated.
\end{abstract}

\pagestyle{myheadings} \thispagestyle{plain} \markboth{Y. WANG AND G. WU}{\sc Refined Bounds on the Number of Distinct Eigenvalues}

\date{ }%

\section{Introduction}
\setcounter{equation}{0}

Low-rank updated matrices are of great importance in real applications \cite{Chu,G. H. Golub}. Specifically, the eigenproblem
with respect to a low-rank updated matrix is of great interest \cite{J. R. Bunch,G. H. Golub,GV,Gu,J. H. Wilkinson,Wu,WW}.
However, most work is devoted to discussing the case of symmetric
low-rank perturbations \cite{J. R. Bunch,G. H. Golub,Gu}, the Jordan form \cite{Moro,Wu}, and the characteristic polynomial  of a low-rank updated matrix \cite{WW}.
Recently, Farrell \cite{P. E. Farrell} gave an upper bound for the number of distinct eigenvalues of arbitrary matrices perturbed by updates of arbitrary rank, which is the central result in \cite{P. E. Farrell}. This result can be utilized to estimate the number of Krylov iterations required for solving a perturbed linear system. Very recently, by separating the spectra into two disjoint sets, Xu presented an improved upper bound in \cite{X.F. Xu}.

Let us briefly introduce some definitions and notations that
will be used in this paper.
Given a matrix $M\in \mathbb{C}^{n\times n}, $ let $\Lambda(M)$ be the set of all distinct eigenvalues of $M$, and let $|\Lambda(M)|$ be the cardinality of the set $\Lambda(M)$.
The {algebraic multiplicity} $m_a (M,\lambda)$ of an eigenvalue $\lambda\in\Lambda(M)$ is the multiplicity of $\lambda$ as a zero of the characteristic
polynomial of $M$. The dimension of the eigenspace of $M$ corresponding to $\lambda$, denoted by $m_g (M,\lambda)$, is called its {geometric multiplicity}.
Recall that the geometric multiplicity of an eigenvalue will never be
greater than its algebraic multiplicity. If $m_g (M,\lambda)<m_a (M,\lambda)$ for some $\lambda\in \Lambda(M)$, then $M$
is called a {defective} matrix. If $m_g (M,\lambda)=m_a (M,\lambda)$ for all $\lambda\in \Lambda(M)$, then $M$ is said to be
{nondefective} or {diagonalizable}. Clearly, $m_g (M,\lambda)\geq 1$ for
all $\lambda\in \Lambda(M)$. If $m_g (M,\lambda)\equiv 1$ for all $\lambda\in \Lambda(M)$, then $M$ is said to be {nonderogatory}; otherwise,
it is referred to as a {derogatory} matrix. We denote by $r(M)$ the rank of the matrix $M$. Let $S_1$ and $S_2$ be two sets, then $S_1\backslash S_2$ stands for the complement of $S_2$ with respect to $S_1$.

The following definition defines the defectivity of an eigenvalue and that of a matrix:
\begin{definition} {\cite{P. E. Farrell}}
The defectivity of an eigenvalue $d(M,\lambda)\geq 0$ is the difference between its algebraic and geometric multiplicities
\begin{equation}\label{211}
d(M,\lambda)=m_a (M,\lambda)-m_g (M,\lambda).
\end{equation}%
And the defectivity of a matrix is the sum of the  defectivities of its eigenvalues
\begin{equation}\label{2.2}
d(M)=\sum\limits_{\lambda\in \Lambda(M)} d(M,\lambda).
\end{equation}%
\end{definition}

As $m_a (M,\lambda)\geq m_g (M,\lambda)$ for all $\lambda\in \Lambda(M)$, we have $d(M)\geq d(M,\lambda)\geq 0$.
Indeed, the defectivity of a matrix can be considered as a quantitative measure of its nondiagonalizability \cite{P. E. Farrell}. Note that a matrix $M$ is diagonalizable if and only if $d(M) = 0$. For completeness, if $\lambda \not\in \Lambda(M)$, we simply set $m_a (M,\lambda)=0$ in this paper. Therefore,
\begin{equation}\label{2.3}
\lambda \not\in \Lambda(M)\Longleftrightarrow m_a (M,\lambda)=0\Longleftrightarrow m_g (M,\lambda)=0.
\end{equation}%
The following result is the central theorem of \cite{P. E. Farrell}. It relates the number of distinct eigenvalues of a matrix to the priori number of distinct eigenvalues.
\begin{theorem}\label{Thm211}\cite[Theorem 1.3]{P. E. Farrell}
Let $A,B\in\mathbb{C}^{n\times n}$. If $C=A+B$, then
\begin{equation}\label{1.1}
|\Lambda(C)|\leq (r(B)+1)|\Lambda(A)|+d(A).
\end{equation}
\end{theorem}

However, the upper bound \eqref{1.1} may not be sharp and even can be invalid in practice \cite{X.F. Xu}.
In order to enhance this result, the following bound was given, which is the main result of \cite{X.F. Xu}.
\begin{theorem}\label{Thm2.2}\cite[Theorem 3.1]{X.F. Xu}
Assume that $A,B\in\mathbb{C}^{n\times n}$ and let $C=A+B$, then we have
\begin{equation}\label{1.2}
|\Lambda(C)|\leq (r(B)+1)|\Lambda(A)|+d(A)-d(C).
\end{equation}
\end{theorem}

The next result plays an important role in estimating the number of Krylov iterations after a rank one update \cite{X.F. Xu}.
\begin{theorem}\label{Thm1.3}\cite[Corollary 4.2]{X.F. Xu}
Suppose that $A \in\mathbb{C}^{n\times n}$ is diagonalizable, $r(B)=1$ and let $C=A+B$. If $C$  is also diagonalizable, then
$|\Lambda(C)|\leq 2|\Lambda(A)|.$ If $C$  is not diagonalizable, then
\begin{equation}\label{1.3}
|\Lambda(C)|\leq 2|\Lambda(A)|-1.
\end{equation}
\end{theorem}

In this paper, we show that all the results \eqref{1.1}, \eqref{1.2} and \eqref{1.3} are not sharp enough and can be overestimated in practice.
Thus, it is necessary to establish new upper bounds on this problem.
We give some refined upper bounds on the number of distinct eigenvalues of a perturbed matrix. Further, we provide some {\it prior} upper bounds that only rely on the information of $A$ and $B$.
The key is to first separate $\Lambda(A)\cup\Lambda(C)$ into three disjoint sets, and then pay special attention to the set $\Lambda(A)\setminus\Lambda(C)$, i.e., the distinct eigenvalues that are in $\Lambda(A)$ but not in $\Lambda(C)$.
Examples show the tightness of our new results, as well as their superiority over those provided in  \cite{P. E. Farrell,X.F. Xu}.
The number of distinct singular values of a matrix after perturbation is also discussed.


\section{The main results}
\setcounter{equation}{0}
In this section, we propose refined bounds on the
number of distinct eigenvalues of a perturbed matrix, which further improves the results given in Theorem \ref{Thm211}--Theorem \ref{Thm1.3}. Inspired by the definition of the derogatory index of a matrix \cite{X.F. Xu}, we can define
the derogatory of an eigenvalue as follows.
\begin{definition}
The derogatory of an eigenvalue $\lambda\in \Lambda(M)$ is defined as
\begin{equation*}
I(M,\lambda)=m_g (M,\lambda)-1.
\end{equation*}%
\end{definition}
Notice that $m_g (M,\lambda)-I(M,\lambda)=1 $ for  all $\lambda\in \Lambda(M)$, we have
\begin{equation}
|\Lambda(M)|=\sum\limits_{\lambda\in \Lambda(M)} \big(m_g (M,\lambda)-I(M,\lambda)\big). \label{21}
\end{equation}%

We are ready to prove the following two lemmas.
\begin{lemma}\label{Lem3.1}
Let $M\in \mathbb{C}^{n\times n}$ and let $S$ be a subset of $\Lambda(M)$.  Then
\begin{equation*}
\sum\limits_{\lambda\in S}m_g(M,\lambda)=n-\sum\limits_{\lambda\in \Lambda(M)\backslash S}  m_g(M,\lambda)-d(M), \label{2.1}
\end{equation*}
where $\Lambda(M)\backslash S$ denotes the complement of $S$ with respect to $\Lambda(M)$.
\end{lemma}
\begin{proof}
Recall that $\sum\limits_{\lambda\in \Lambda(M)} m_a (M,\lambda)=n$, and we have from \eqref{211} and \eqref{2.2} that
\begin{eqnarray*}
 d(M) &=& n-\sum\limits_{\lambda\in \Lambda(M)} m_g (M,\lambda)\\
       &=& n-\sum\limits_{\lambda\in  S} m_g (M,\lambda)-\sum\limits_{\lambda\in \Lambda(M)\backslash S} m_g (M,\lambda).
\end{eqnarray*}%
Hence,
\begin{eqnarray*}
\sum\limits_{\lambda\in S}m_g(M,\lambda)=n-\sum\limits_{\lambda\in \Lambda(M)\backslash S}  m_g(M,\lambda)-d(M).
\end{eqnarray*}
\end{proof}

\begin{lemma}\label{Thm3.1}
Let $A, B\in \mathbb{C}^{n\times n}$ and $C=A+B$. Denote $S_1=\Lambda(A)\cap \Lambda(C), S_2=\Lambda(C)\setminus S_1$, and $S_3=\Lambda(A)\setminus S_1$.
Then
\begin{equation}\label{3.1}
|\Lambda(C)|\leq (r(B)+1)|\Lambda(A)|+d(A)-d(C)-N(A,B,C),
\end{equation}
where
\begin{equation}\label{3.111}
N(A,B,C)=|S_3|+\sum\limits_{\lambda\in S_3}\big(r(B)-m_g(A,\lambda)\big)+\sum\limits_{\lambda\in S_2}I(C,\lambda)
\end{equation}
is a nonnegative number.
\end{lemma}
\begin{proof}
We note that
\begin{equation}
|\Lambda(C)|=|S_1|+|S_2|,  \label{3.2}
\end{equation}
\begin{equation}
|S_1|=|\Lambda(A)|-|S_3|,  \label{3.3}
\end{equation}
and it follows from \eqref{21} that
 \begin{equation}
|S_2|=\sum\limits_{\lambda\in S_2} \big(m_g(C,\lambda)-I(C,\lambda)\big).  \label{3.4}
\end{equation}
As $S_1=\Lambda(C)\backslash S_2$, we have from Lemma \ref{Lem3.1} that
\begin{eqnarray*}
\sum\limits_{\lambda\in S_2} m_g(C,\lambda)
  = n-\sum\limits_{\lambda\in S_1} m_g(C,\lambda)-d(C).
\end{eqnarray*}%
Moreover, we obtain from \cite[(1.7c)]{P. E. Farrell} that
\begin{equation}\label{3.7}
m_g(A,\lambda)-m_g(C,\lambda)\leq r(B).
\end{equation}
Thus,
\begin{eqnarray*} 
\sum\limits_{\lambda\in S_2} m_g(C,\lambda)
   &\leq &    n+\sum\limits_{\lambda\in S_1}\big(r(B)- m_g(A,\lambda)\big)-d(C) \nonumber\\
   &= &    r(B)|S_1|+n-\sum\limits_{\lambda\in S_1}m_g(A,\lambda)-d(C).
\end{eqnarray*}%

On the other hand, since $S_3=\Lambda(A)\backslash S_1$, we obtain from Lemma \ref{Lem3.1} that
\begin{equation*}
 n-\sum\limits_{\lambda\in S_1}m_g(A,\lambda)=d(A)+\sum\limits_{\lambda\in S_3}m_g(A,\lambda).
\end{equation*}%
Hence,
\begin{equation}\label{3.9}
\sum\limits_{\lambda\in S_2} m_g(C,\lambda)
   \leq   r(B)|S_1|+d(A)+\sum\limits_{\lambda\in S_3}m_g(A,\lambda)-d(C). \\
\end{equation}

Combining \eqref{3.2}--\eqref{3.9}, we arrive at
{\small\begin{eqnarray*}
|\Lambda(C)| &=& |S_1|+|S_2|=|S_1|+\sum\limits_{\lambda\in S_2} \big(m_g(C,\lambda)-I(C,\lambda)\big)\\
             &\leq&|S_1|+r(B)|S_1|+d(A)+\sum\limits_{\lambda\in S_3}m_g(A,\lambda)-d(C)-\sum\limits_{\lambda\in S_2} I(C,\lambda)\\
             &=&(r(B)+1)(|\Lambda(A)|-|S_3|)+d(A)+\sum\limits_{\lambda\in S_3}m_g(A,\lambda)-d(C)-\sum\limits_{\lambda\in S_2} I(C,\lambda)\\
             &=&(r(B)+1)|\Lambda(A)|-|S_3|-r(B)|S_3|+d(A)+\sum\limits_{\lambda\in S_3}m_g(A,\lambda)-d(C)-\sum\limits_{\lambda\in S_2} I(C,\lambda)\\
             &=&(r(B)+1)|\Lambda(A)|+d(A)-d(C)-|S_3|-\sum\limits_{\lambda\in S_3}\big(r(B)-m_g(A,\lambda)\big)-\sum\limits_{\lambda\in S_2} I(C,\lambda)\\
             &=&(r(B)+1)|\Lambda(A)|+d(A)-d(C)-N(A,B,C).
\end{eqnarray*}}%
Furthermore, we note from \eqref{3.7} and \eqref{2.3} that $m_g (C,\lambda)=0,~\forall\lambda\in S_3$. Thus,
\begin{eqnarray*}
 r(B)-m_g(A,\lambda)\geq 0, \ \ \ \  \forall \lambda\in S_3,
\end{eqnarray*}%
from which we get $N(A,B,C)\geq 0$.
\end{proof}

\begin{remark}
Notice that all the three terms in $N(A,B,C)$ are greater than or equal to zero. Therefore, the new bound \eqref{3.1} is better than those given in \eqref{1.1} and \eqref{1.2}.
\end{remark}

\begin{example}
We want to show the sharpness of \eqref{3.1}. 
Consider the matrices
$$
A=\left[
      \begin{array}{ccccc}
        1 & 0 & 0 & 0 & 0 \\
        0 & 0 & 0 & 0 & 0 \\
        0 & 0 & 1 & 0 & 0\\
        0 & 0 & 0 & 0 & 0\\
        0 & 0 & 0 & 0 & 2\\
      \end{array}
    \right],\quad
    B=\left[
      \begin{array}{ccccc}
        1 & 1 & 0 & 0  & 0 \\
        1 & 1 & 0 & 0  & 0 \\
        0 & 0 & 1 & 1  & 0 \\
        0 & 0 & 1 & 1  & 0 \\
        0 & 0 & 0 & 0  & 1 \\
      \end{array}
    \right],
    $$
and
 \begin{eqnarray*}
 C=A+B=\left[
      \begin{array}{ccccc}
        2 & 1 & 0 & 0 & 0\\
        1 & 1 & 0 & 0 & 0\\
        0 & 0 & 2 & 1 & 0\\
        0 & 0 & 1 & 1 & 0\\
        0 & 0 & 0 & 0 & 3\\
      \end{array}
    \right].
    \end{eqnarray*}%
It is straightforward to show that $d(A)=d(C)=0,~\Lambda(A)=\{0,1,2\},~\Lambda(C)=\{\frac{3-\sqrt{5}}{2},~\frac{3+\sqrt{5}}{2},~3\}$, and $r(B)=3$. Thus, $|\Lambda(A)|=|\Lambda(C)|=3,~S_1=\emptyset,~S_2=\Lambda(C),~S_3=\Lambda(A)$, and
$m_g(A,0)=m_g(A,1)=2,~m_g(A,2)=1,~I(C,\frac{3-\sqrt{5}}{2})=I(C,\frac{3+\sqrt{5}}{2})=1,~I(C,3)=0$. So we have
$$
|S_3|=3,\quad \sum\limits_{\lambda\in S_3}\big(r(B)-m_g(A,\lambda)\big)=(3-2)+(3-2)+(3-1)=4,\quad \sum\limits_{\lambda\in S_2} I(C,\lambda)=2,
$$
and
$$
N(A,B,C)=9.
$$
For this example, both \eqref{1.1} and \eqref{1.2} yield
$$
|\Lambda(C)|\leq \big(r(B)+1\big)|\Lambda(A)|=12,
$$
which is invalid because the order of $C$ is 5. As a comparison, \eqref{3.1} gives
\begin{eqnarray*}
|\Lambda(C)|\leq\big(r(B)+1\big)|\Lambda(A)|-N(A,B,C)=12-9=3.
    \end{eqnarray*}%
As $|\Lambda(C)|=3$, \eqref{3.1} is sharp.
\end{example}

However, the bound given in Lemma \ref{Thm3.1} is too complicated to use in practice. Indeed, it involves $m_g(A,\lambda)$ for each $\lambda\in S_3$ and $I(C,\lambda)$ for each $\lambda\in S_2$. We want to provide a more practical bound, and the idea is using $|S_3|=|\Lambda(A)\backslash \Lambda(C)|$ to take the place of $N(A,B,C)$ in \eqref{3.1}.
We can present the first main theorem in this paper.
\begin{theorem}\label{Cor2.4}
Let $A, B\in \mathbb{C}^{n\times n}$ and $C=A+B$.
Then
\begin{equation}\label{3.11}
|\Lambda(C)|\leq \big(r(B)+1\big)|\Lambda(A)|+d(A)-d(C)-|\Lambda(A)\backslash \Lambda(C)|.
\end{equation}
Specifically, if the spectra of $A$ and $C$ are disjoint, then we have
\begin{equation}\label{3.12}
|\Lambda(C)|\leq r(B)|\Lambda(A)|+d(A)-d(C).
\end{equation}
\end{theorem}
\begin{proof}
The inequality \eqref{3.11} follows directly from the fact that all the three terms in \eqref{3.111} are greater than or equal to zero.
For \eqref{3.12}, we have
$|\Lambda(A)\backslash \Lambda(C)|=|\Lambda(A)|$ if the spectra of $A$ and $C$ are disjoint. By \eqref{3.11},
\begin{eqnarray*}
|\Lambda(C)| &\leq & (r(B)+1)|\Lambda(A)|+d(A)-d(C)-|\Lambda(A)\backslash \Lambda(C)|\\
            &= & (r(B)+1)|\Lambda(A)|+d(A)-d(C)-|\Lambda(A)|\\
            &=& r(B)|\Lambda(A)|+d(A)-d(C).
\end{eqnarray*}
\end{proof}
\begin{remark}
Theorem \ref{Cor2.4} indicates that the upper bounds given in \eqref{1.1} and \eqref{1.2}
can be enhanced substantially in many case. The key idea is to first separate $\Lambda(A)\cup\Lambda(C)$ into three disjoint sets, and then pay special attention to the set $\Lambda(A)\setminus\Lambda(C)$.
Obviously, our new bounds are better than those given in \eqref{1.1} and \eqref{1.2}.
Note that $|S_3|>0$ {\it if and only if} there is at least an element in $\Lambda(A)$ but not in $\Lambda(C)$.
Specifically, $|S_3|=|\Lambda(A)|$ as the eigenvalues of $A$ and $C$ are disjoint. 
\end{remark}

\begin{example}
In \eqref{3.12}, it requires that the eigenvalues of $A$ and those of its low-rank update $C$, are disjoint with each other. In this example, we try illustrate that this condition is not stringent and even trivial in practice, at least for randomly generated matrices $A$ and $B$:

Given $n$, the size of the matrix $A$, $k$, the rank of the perturbation matrix $B$, and $m$, the number of tests, we run the following MATLAB code. Both the test matrices $A$ and the low-rank matrices $B$ are generated by using the MATLAB command {\tt randn}, where {\tt randn(M,N)} returns an M-by-N matrix containing pseudorandom values drawn
from the standard normal distribution \cite{MATLAB}.
{\it
\begin{enumerate}
  \item  ~ function num=dinstincteigen(n,k,m)
  \item  ~{\bf for } i=1:m ~~\% Run the test for $m$ times
  \item  ~~~~~ A=randn(n,n);
  \item  ~~~~~ D1=unique(eig(A));~~\% $\Lambda(A)$
  \item  ~~~~~ B=randn(n,k)*randn(k,n); ~~\% Generate a random rank-k matrix
  \item  ~~~~~ C=A+B;
  \item  ~~~~~ D2=unique(eig(C));~~\% $\Lambda(C)$
  \item  ~~~~~ D=setdiff(D1,D2);~~\%  $\Lambda(A)\backslash \Lambda(C)$
  \item  ~~~~~ num=size(D,1);~~ \% $S_3=|\Lambda(A)\backslash \Lambda(C)|$
  \item  ~{\bf end }
\end{enumerate}
}
We run this code with $n=1000$, $k=1,2,\ldots,5$, and $m=1000$. In other words, we run the test for $mk=5000$ times altogether.
The numerical results show that $num=1000$ for all the $5000$ runs. This implies that we often have $\Lambda(A)\cap\Lambda(C)=\emptyset$, at least for randomly generated matrices $A$ and $B$.
Therefore, the condition $|\Lambda(A)\backslash \Lambda(C)|>0$ is trivial and $|\Lambda(A)\backslash \Lambda(C)|=|\Lambda(A)|$ is not stringent in practice.
%
\end{example}


Next, we show the sharpness of \eqref{3.11} and \eqref{3.12}, and demonstrate the superiority of them over \eqref{1.1} and \eqref{1.2}.
\begin{example}
Consider
$$
A=\left[
      \begin{array}{ccc}
        0 & 0 & -1 \\
        -1 & 0  & -1  \\
        -1 & -1  &  1  \\
      \end{array}
    \right],\quad
    B=\left[
      \begin{array}{ccc}
        1 & 1 & 1 \\
         1 & 1 & 1 \\
          1 & 1 & 1 \\
      \end{array}
    \right],
    $$
and
 \begin{eqnarray*}
 C=A+B=\left[
      \begin{array}{ccc}
        1 & 1 & 0 \\
        0 & 1 & 0 \\
        0 & 0 & 2 \\
      \end{array}
    \right].
    \end{eqnarray*}%
We have $\Lambda(A)=\{-1.2470, 0.4450,1.8019\},~d(A)=0,~ \Lambda(C)=\{1,2\},~ d(C)=1,~ r(B)=1$, and $|S_3|=|\Lambda(A)|=3$. For this example, \eqref{1.1} and \eqref{1.2} give us
$$
|\Lambda(C)|\leq \big(r(B)+1\big)|\Lambda(A)|+d(A)=6,
$$
and
$$
|\Lambda(C)|\leq \big(r(B)+1\big)|\Lambda(A)|+d(A)-d(C)=5,
$$
respectively. Obviously, the above two bounds are invalid since the matrix $C$ is of order 3. As a comparison, both \eqref{3.11} and \eqref{3.12} yield
\begin{eqnarray*}
|\Lambda(C)|\leq r(B)|\Lambda(A)|+d(A)-d(C)=3-1=2,
    \end{eqnarray*}%
which is sharp for the sake of $|\Lambda(C)|=2$.
\end{example}

However, all the results provided in Theorem \ref{Thm2.2}, Lemma \ref{Thm3.1} and Theorem \ref{Cor2.4}
need to know the eigen-information on the perturbed matrix $C$, which is not available {\it a prior}.
Thus, we aim to provide some {\it prior} upper bounds that only rely on the information of $A$ and $B$. Similar to \cite{P. E. Farrell}, we assume that the information on $|\Lambda(A)|$ is known in advance. We are in a position to give the second main theorem in this paper.
\begin{theorem}\label{Cor2.6} Suppose that $A \in \mathbb{C}^{n\times n}$ is diagonalizable and
$C=A+B$. Then
{\rm (i)} If $C$ is diagonalizable and $|\Lambda(A)\backslash \Lambda(C)|\geq 1$, we have
\begin{equation}\label{3.15}
|\Lambda(C)|\leq (r(B)+1)|\Lambda(A)|-1.
\end{equation}
Specifically, if $\Lambda(A)\cap\Lambda(C)=\emptyset$, then
\begin{equation}\label{41}
|\Lambda(C)|\leq r(B)|\Lambda(A)|.
\end{equation}
{\rm (ii)} If $C$ is non-diagonalizable and $|\Lambda(A)\backslash \Lambda(C)|\geq 1$, then
\begin{equation}\label{3.13}
|\Lambda(C)|\leq (r(B)+1)|\Lambda(A)|-2.
\end{equation}
Specifically, if $\Lambda(A)\cap\Lambda(C)=\emptyset$, then
\begin{equation}\label{3.14}
|\Lambda(C)|\leq r(B)|\Lambda(A)|-1.
\end{equation}
\end{theorem}
\begin{proof}
Recall that $d(A)=d(C)=0$ if $A$ and $C$ are diagonalizable. So \eqref{3.15} follows directly from \eqref{3.11}.
Specifically, if $\Lambda(A)\cap\Lambda(C)=\emptyset$, we have $|\Lambda(A)\backslash \Lambda(C)|=|\Lambda(A)|$, and thus \eqref{41} holds.

On the other hand, if $A$ is diagonalizable while $C$ is non-diagonalizable, we have $d(A)=0$ and $d(C)\geq 1$. From \eqref{3.11}
and the fact that $|\Lambda(A)\backslash \Lambda(C)|\geq 1$, we get
\begin{eqnarray*}
|\Lambda(C)| \leq  (r(B)+1)|\Lambda(A)|-1-1
            =(r(B)+1)|\Lambda(A)|-2.
\end{eqnarray*}
Specifically, if $C$ is non-diagonalizable and $\Lambda(A)\cap\Lambda(C)=\emptyset$, we obtain \eqref{3.14} from \eqref{3.12}.
\end{proof}
\begin{remark}
Theorem \ref{Cor2.6} is an analogy of Theorem \ref{Thm1.3}, which plays an important role in the estimate for the number of Krylov iterations after low-rank update.
Compared with Theorem \ref{Thm1.3}, our new results are tighter and improve those given in \cite[Corollary 4.2]{X.F. Xu} significantly: First, when $C$ is diagonalizable, $r(B)=1$, and $|\Lambda(A)\backslash \Lambda(C)|\geq 1$ {\rm(}which is trivial in practice{\rm)}, it follows from \eqref{3.15} that
$$
|\Lambda(C)|\leq 2|\Lambda(A)|-1< 2|\Lambda(A)|.
$$
Second, if $C$ is non-diagonalizable, $r(B)=1$ and $\Lambda(A)\cap\Lambda(C)=\emptyset$ {\rm(}which is not stringent in practice{\rm)}, we have
$$
|\Lambda(C)|\leq |\Lambda(A)|-1,
$$
which is much smaller than
$2|\Lambda(A)|-1$; see \eqref{1.3}.

On the other hand, the upper bound given in Theorem \ref{Thm211} involves $d(A)$, i.e., the sum of the  defectivities of all the eigenvalues of $A$, which is difficult to evaluate. In practice, we can use 
$
(r(B)+1)|\Lambda(A)|-1
$ 
as an upper bound to $|\Lambda(C)|$, if we only have $|\Lambda(A)|$ and $r(B)$ at hand and there is no other information available. Another reason is that it is the upper bound of \eqref{3.15}--\eqref{3.14}, and it is much better than the right-hand side of \eqref{1.1}.
\end{remark}

\begin{example}
In this example, we try to show the sharpness of \eqref{3.15}--\eqref{3.14}. 

{\rm (i)}~~
Consider the two $n$-by-$n$ matrices
$$
A=\left[
      \begin{array}{ccccc}
        0 & 0 & 0 & \cdots & 0 \\
        0 & 1 & 0 & \cdots & 0 \\
        0 & 0 & 1 & \cdots & 0 \\
        \vdots & \vdots & \vdots & \ddots & \vdots \\
       0 & 0 & 0 & \cdots & 1 \\
      \end{array}
    \right],\quad
    B=\left[
      \begin{array}{ccccc}
        1 & 1 & 0 & \cdots & 0 \\
        1 & 1 & 0 & \cdots & 0 \\
        0 & 0 & 0 & \cdots & 0 \\
        \vdots & \vdots & \vdots & \ddots & \vdots \\
       0 & 0 & 0 & \cdots & 0 \\
      \end{array}
    \right].
    $$
Then
 \begin{eqnarray*}
 C=A+B=\left[
      \begin{array}{ccccc}
        1 & 1 & 0 & \cdots & 0 \\
        1 & 2 & 0 & \cdots & 0 \\
        0 & 0 & 1 & \cdots & 0 \\
        \vdots & \vdots & \vdots & \ddots & \vdots \\
       0 & 0 & 0 & \cdots & 1 \\
      \end{array}
    \right].
    \end{eqnarray*}%
 Note that both $A$ and $C$ are diagonalizable since they are symmetric.  Moreover, $\Lambda(A)=\{0,1\},~\Lambda(C)=\{1,~\frac{3-\sqrt{5}}{2},~\frac{3+\sqrt{5}}{2}\}$, and $r(B)=1$. Thus, $|\Lambda(A)|=2,~|\Lambda(C)|=3$, and $|\Lambda(A)\backslash \Lambda(C)|=1$, and \eqref{3.15} gives
\begin{eqnarray*}
|\Lambda(C)|\leq\big(r(B)+1\big)|\Lambda(A)|-1=4-1=3.
    \end{eqnarray*}%
We see \eqref{3.15} is sharp because of $|\Lambda(C)|=3$.
%

{\rm (ii)}~~Consider
$$
A=\left[
      \begin{array}{ccc}
        1 & 0 & 0 \\
        0 & 2  & 0  \\
        0 & 0  &  3 \\
      \end{array}
    \right],\quad
    B=\left[
      \begin{array}{ccc}
        1 & 1 & 1 \\
         1 & 1 & 1 \\
          1 & 1 & 1 \\
      \end{array}
    \right],
    $$
and
 \begin{eqnarray*}
 C=A+B=\left[
      \begin{array}{ccc}
        2 & 1 & 1 \\
        1 & 3 & 1 \\
        1 & 1 & 4 \\
      \end{array}
    \right].
    \end{eqnarray*}%
In this example, both $A$ and $C$ are diagonalizable, moreover, $\Lambda(A)=\{1, 2, 3\},~r(B)=1$, $\Lambda(C)=\{1.3249,  2.4608, 5.2143\}$, and $\Lambda(A)\cap\Lambda(C)=\emptyset$. We have from \eqref{41} that
\begin{eqnarray*}
|\Lambda(C)|\leq r(B)|\Lambda(A)|=3,
\end{eqnarray*}%
which is sharp for the sake of $|\Lambda(C)|=3$.
%

{\rm (iii)}~~
We try to show the sharpness of \eqref{3.13}. 
Consider the two  n-by-n matrices
$$
A=\left[
      \begin{array}{ccccc}
        1 & 0 & 0 & \cdots & 0 \\
        0 & 2 & 0 & \cdots & 0 \\
        0 & 0 & 2 & \cdots & 0 \\
        \vdots & \vdots & \vdots & \ddots & \vdots \\
       0 & 0 & 0 & \cdots & 2 \\
      \end{array}
    \right],\quad
    B=\left[
      \begin{array}{ccccc}
        2 & 0 & 1 & \cdots & 0 \\
        2 & 0 & 1 & \cdots & 0 \\
        0 & 0 & 0 & \cdots & 0 \\
        \vdots & \vdots & \vdots & \ddots & \vdots \\
       0 & 0 & 0 & \cdots & 0 \\
      \end{array}
    \right].
    $$
Then
 \begin{eqnarray*}
 C=A+B=\left[
      \begin{array}{ccccc}
        3 & 0 & 1 & \cdots & 0 \\
        2 & 2 & 1 & \cdots & 0 \\
        0 & 0 & 2 & \cdots & 0 \\
        \vdots & \vdots & \vdots & \ddots & \vdots \\
       0 & 0 & 0 & \cdots & 2 \\
      \end{array}
    \right].
    \end{eqnarray*}%
  Obviously, $A$ is diagonalizable, $\Lambda(A)=\{1,2\},~ r(B)=1, ~\Lambda(C)=\{2, 3\}$ and $m_a(C,2)=n-1$. Thus, $|\Lambda(A)|=2,~|\Lambda(C)|=2$, and $|\Lambda(A)\backslash \Lambda(C)|=1$. It turns out that $C$ is non-diagonalizable. Indeed,
 \begin{eqnarray*}
 C-2I=\left[
      \begin{array}{ccccc}
        1 & 0 & 1 & \cdots & 0 \\
        2 & 0 & 1 & \cdots & 0 \\
        0 & 0 & 0 & \cdots & 0 \\
        \vdots & \vdots & \vdots & \ddots & \vdots \\
       0 & 0 & 0 & \cdots & 0 \\
      \end{array}
    \right],
    \end{eqnarray*}%
    where $I$ is the identity matrix.
Therefore, $r(C-2I)=2$, implying that $m_g(C,2)=n-2< m_a(C,2)=n-1$.
It follows from \eqref{3.13} that
\begin{eqnarray*}
|\Lambda(C)|\leq\big(r(B)+1\big)|\Lambda(A)|-2=4-2=2.
    \end{eqnarray*}%
As $|\Lambda(C)|=2$, \eqref{3.13} is sharp.
%

{\rm (iv)}~~
Consider
$$
A=\left[
      \begin{array}{ccc}
        1 & -1 & 1 \\
        -1 & 0  & 1  \\
        -1 & -2  &  1 \\
      \end{array}
    \right],\quad
    B=\left[
      \begin{array}{ccc}
        1 & 2 & -1 \\
         1 & 2 & -1 \\
          1 & 2 & -1 \\
      \end{array}
    \right],
    $$
and
 \begin{eqnarray*}
 C=A+B=\left[
      \begin{array}{ccc}
        2 & 1 & 0 \\
        0 & 2 & 0 \\
        0 & 0 & 0 \\
      \end{array}
    \right].
    \end{eqnarray*}%
It is seen that $A$ is diagonalizable while $C$ is not. We have $\Lambda(A)=\{1.6506, 0.1747 + 1.5469{\bf i}, 0.1747 - 1.5469{\bf i}\}$, $r(B)=1$, and $\Lambda(A)\cap\Lambda(C)=\emptyset$, where ${\bf i}$ is the imaginary unit. For this example,  \eqref{3.14} gives 
\begin{eqnarray*}
|\Lambda(C)|\leq r(B)|\Lambda(A)|-1=2.
    \end{eqnarray*}%
Note that $|\Lambda(C)|=2$,  \eqref{3.14} is sharp.
\end{example}


The singular values of low-rank updated matrices are of great interest in many applications \cite{Chu}.
Finally, we give the following result on the number of distinct singular values of a matrix after perturbation.
\begin{corollary}
Let $A, B\in \mathbb{C}^{n\times n}$ and $C=A+B.$ Then
\begin{equation}\label{4.1}
|\sigma(C)|\leq (2r(B)+1)|\sigma(A)|-|\sigma(A)\backslash \sigma(C)|.
\end{equation}
where $\sigma(\cdot)$ denotes the set of distinct singular values of a matrix, and $|\sigma(\cdot)|$ denotes the cardinality of the set $\sigma(\cdot)$ .
\end{corollary}
\begin{proof}
Notice that
\begin{eqnarray*}
C^*C=(A+B)^*(A+B)=A^*A+A^*B+B^*(A+B),
\end{eqnarray*}%
where $C^*$ represents the conjugate transpose of $C$.
From $r\big(A^*B+B^*(A+B)\big)\leq 2r(B)$, $d(A^*A)=d(C^*C)=0$ and \eqref{3.11}, we obtain
\begin{eqnarray*}
|\Lambda(C^*C)| &\leq&  \big(r\big(A^*B+B^*(A+B)\big)+1\big)|\Lambda(A^*A)|-|\Lambda(A^*A)\backslash \Lambda(C^*C)|\\
               &\leq&  \big(2r(B)+1\big)|\Lambda(A^*A)|-|\Lambda(A^*A)\backslash \Lambda(C^*C)|.
\end{eqnarray*}%
Or equivalently,
\begin{equation*}
|\sigma(C)|\leq \big(2r(B)+1\big)|\sigma(A)|-|\sigma(A)\backslash \sigma(C)|.
\end{equation*}
\end{proof}


\begin{example}
In this example, we demonstrate the sharpness of \eqref{4.1}.
Consider
$$
A=\left[
      \begin{array}{ccccc}
        1 & 0 & 0 & 0 & 0 \\
         0 & 1 & 0 & 0 & 0 \\
          0 & 0 & 1 & 0 & 0 \\
           0 & 0 & 0 & 1 & 0 \\
            0 & 0 & 0 & 0 & 1 \\
      \end{array}
    \right],\quad
    B=\left[
      \begin{array}{ccccc}
        0 & 0 & 0 & 0 & 1 \\
         0 & 0 & 0 & 0 & 1 \\
          0 & 0 & 0 & 0 & 1 \\
           0 & 0 & 0 & 0 & 1 \\
            0 & 0 & 0 & 0 & 1 \\
      \end{array}
    \right].
    $$
It is easy to check that
 \begin{eqnarray*}
 C=A+B=\left[
      \begin{array}{ccccc}
        1 & 0 & 0 & 0 & 1 \\
         0 & 1 & 0 & 0 & 1 \\
          0 & 0 & 1 & 0 & 1 \\
           0 & 0 & 0 & 1 & 1 \\
            0 & 0 & 0 & 0 & 2 \\
      \end{array}
    \right],
    \end{eqnarray*}%
$\sigma(C)=\{1, 2.9208, 0.6847\}$ and $|\sigma(A)\backslash \sigma(C)|=0.$  We have from \eqref{4.1} that
\begin{eqnarray*}
|\sigma(C)|\leq\big(2r(B)+1\big)|\sigma(A)|-|\sigma(A)\backslash \sigma(C)|=3-0=3.
    \end{eqnarray*}%
Recall that $|\sigma(C)|=3$, so \eqref{4.1} is sharp.
\end{example}

\bigskip

\section*{Acknowledgments}
We would like to thank  Dr. Hongkui Pang for careful reading the manuscript and helpful discussions.

\newpage


\end{document}